\documentclass[11pt]{amsart}
\usepackage{hyperref}
\usepackage{anysize} \marginsize{1.3in}{1.3in}{1in}{1in}
\usepackage{amsfonts}
\usepackage{mathabx}

\usepackage[all,cmtip]{xy}
\usepackage{amsmath}
\usepackage{amsthm}
\usepackage{amssymb}
\usepackage{mathrsfs}
\usepackage{enumitem}
\usepackage{tikz-cd}
\usepackage{tikz}
\usepackage[normalem]{ulem}

\newtheorem{thm}{Theorem}[section]

\newtheorem{prop}[thm]{Proposition}
\newtheorem{lemma}[thm]{Lemma}

\newtheorem{lem}[thm]{Lemma}
\newtheorem{cor}[thm]{Corollary}

\theoremstyle{definition}
\newtheorem{defn}[thm]{Definition}
\newtheorem{notn}[thm]{Notation}
\newtheorem*{definition*}         {Definition}

\theoremstyle{remark}

\newtheorem{remark}[thm]{Remark}

\newcommand{\End}{\mathrm{End}}
\newcommand{\Hom}{\mathrm{Hom}}

\newcommand{\bH}{\mathbb{H}}

\newcommand{\bZ}{\mathbb{Z}}
\newcommand{\Z}{\mathbb{Z}}
\newcommand{\Q}{\mathbb{Q}}
\newcommand{\bQ}{\mathbb{Q}}
\newcommand{\bR}{\mathbb{R}}
\newcommand{\bC}{\mathbb{C}}
\newcommand{\C}{\mathbb{C}}

\newcommand{\cB}{\mathcal{B}}
\newcommand{\cO}{\mathcal{O}}
\newcommand{\cD}{\mathcal{D}}

\newcommand{\cF}{\mathcal{F}}

\newcommand{\cL}{\mathcal{L}}
\newcommand{\cM}{\mathcal{M}}

\newcommand{\cS}{\mathcal{S}}

\newcommand{\cV}{\mathcal{V}}
\newcommand{\cW}{\mathcal{W}}

\newcommand{\bfG}{\mathbf{G}}
\newcommand{\bfH}{\mathbf{H}}
\newcommand{\bfU}{\mathbf{U}}
\newcommand{\bfGL}{\mathbf{GL}}

\newcommand{\bfS}{\mathbf{S}}

\newcommand{\arrow}{\rightarrow}

\newcommand{\Aut}{\mathrm{Aut}}

\newcommand{\ad}{\mathrm{ad}}
\newcommand{\Ad}{\mathrm{Ad}}

\newcommand{\Res}{\mathrm{Res}}

\renewcommand{\phi}{\varphi}
\def\ra{\rightarrow}
\def\Hdg{\operatorname{Hdg}}

\def\id{\operatorname{id}}
\def\GL{\operatorname{GL}}
\def\SL{\operatorname{SL}}
\def\sl{\operatorname{sl}}

\def\Hom{\operatorname{Hom}}

\def\Res{\operatorname{Res}}

\def\Gr{\operatorname{Gr}}
\def\Ext{\operatorname{Ext}}

\def\End{\operatorname{End}}

\def\SL{\operatorname{SL}}

\def\Ad{\operatorname{Ad}}
\def\ad{\operatorname{ad}}

\def\Aut{\operatorname{Aut}}

\def\MT{\mathbf{MT}}
\def\NL{\operatorname{NL}}
\def\cNL{\mathcal{NL}}

\newcommand{\alg}{\mathrm{alg}}

\newcommand{\an}{\mathrm{an}}

\renewcommand{\bar}[1]{\overline{#1}}

\usepackage{amscd,amssymb,amsmath}
\usepackage{color}

\title{Definability of mixed period maps} 

 \author[B. Bakker]{Benjamin Bakker}
\address{\noindent B. Bakker:  Dept. of Mathematics, University of Georgia, Athens, USA.}
\email{bakker@math.uga.edu}

\author[Y. Brunebarbe]{Yohan Brunebarbe}
\address{\noindent Y. Brunebarbe:  Dept. of Mathematics, Univ. Bordeaux, Talence, France.}
\email{yohan.brunebarbe@math.u-bordeaux.fr}

\author[N. Klingler]{Bruno Klingler}
\address{\noindent B. Klingler:  Dept. of Mathematics, Humboldt
  Universit\"{a}t, Berlin, Germany.}
\email{bruno.klingler@math.hu-berlin.de}

\author[J. Tsimerman]{Jacob Tsimerman}
\address{\noindent J. Tsimerman:  Dept. of Mathematics, University of Toronto, Toronto, Canada.}
\email{jacobt@math.toronto.edu}

\thanks{B.B. was partially supported by NSF grants DMS-1702149 and DMS-1848049. }

\begin{document}
\begin{abstract}We equip integral graded-polarized mixed period spaces
  with a natural $\bR_{\alg}$-definable analytic structure, and prove
  that any period map associated to an admissible variation of
  integral graded-polarized mixed Hodge structures is definable in
  $\bR_{\an,\exp}$ with respect to this structure. As a consequence we
  reprove that the zero loci of admissible normal functions are
  algebraic. \end{abstract} 
\maketitle


\section{introduction}

\subsection{Summary}

The purpose of this paper is to continue the development of
o-minimality as a natural setting for the study of Hodge theory. In
\cite{bkt} it was shown that the moduli of integral polarized pure
Hodge structures---known as \emph{period spaces}---admit natural
structures of definable analytic spaces, in such a way that all period
maps from algebraic varieties are definable. The general functorial
setting of definable analytic spaces was studied in \cite{bbt1}. The
purpose of this article is to extend this technology to the setting of mixed Hodge structures. 

One complication that enters when studying variations of mixed Hodge
  structures (VMHS) is that one must additionally restrict to
\emph{admissible} ones in the sense of Steenbrink-Zucker
  and Kashiwara, instead of just ones that are
holomorphic and Griffiths transverse. This apparent complication,
crucial for the internal coherence of Hodge theory as developed in
the theory of Hodge modules \cite{saito}, fits perfectly with
the o-minimal setting. For any smooth complex algebraic variety $S$
the (not necessarily admissible) VMHSs extensions of $\Z_{S^{\an}}(0)$ by
$\Z_{S^{\an}}(1)$ are parametrized by $\Gamma(S, \cO_{S^{\an}}^*)$,
corresponding to holomorphic period maps $ \varphi: S^\an \arrow
\Ext^1_{\Z MHS}(\Z(0),\Z(1))\cong\C^*$ from $S^\an$ to the mixed period
space $\C^*$. For $S = \mathbb{A}^1$ the period map $\exp:
\C \arrow \C^*$ cannot possibly be definable in any o-minimal
structure; however the VMHS on $\mathbb{A}^1$ it defines is not admissible.

Unlike in the pure case, there is some ambiguity in the choice of definable structure. Indeed, if we think only of the ``unipotent" fibers\footnote{Formally, the fibers of the map from the mixed period space to the product of the pure period spaces corresponding to taking the associated graded variation.} we end up with a quotient of unipotent groups, for which there are many choices of which
definable structure---this is already the case for $\C^*$. The definable structure appearing in Theorem \ref{definability of period maps} is built using the ``$\sl_2$" real splitting (also known as the ``canonical" real splitting), but it is not inconceivable to us that one could use other natural definable structures and retain our main results.

\subsection{Results}

In section \S3 we equip any graded-polarized integral mixed period space $\Gamma\backslash\cM$ with the structure of a $\bR_{\alg}$-definable analytic space which is functorial with respect to morphisms of mixed period spaces (see Theorem \ref{thm def Hodge var}). Our main result is the following:

\begin{thm}[\ref{definability of period maps}]  Let
  $\Gamma\backslash\cM$ be a graded-polarized integral mixed period
  space equipped with the $\bR_{\alg}$-definable structure associated
  to the $\sl_2$-splitting.  Let $S$ be a reduced complex algebraic
  space and $\phi : S \ra\Gamma\backslash\cM$ an admissible period
  map. 
Then $\phi$ is $\bR_{\an, \exp}$-definable.
\end{thm}

It should be noted that the work of Brosnan--Pearlstein \cite{bp} building on the mixed $\SL_2$-orbit theorem of Kato--Nakayama--Usui \cite{knu} is a key ingredient in our proofs, giving the necessary boundedness statement for us to prove definability.

In \cite{bkt} we recovered as an immediate corollary of the definability of the
period map for pure VHS the algebraicity of the corresponding Hodge
loci proven in \cite{alghodge}. Similarly as an immediate corollary of Theorem \ref{definability of period maps}
we recover the algebraicity of (possibly non-reduced) mixed Hodge loci (in particular the zero-loci of
admissible normal functions) obtained in \cite{BP091},  \cite{BP092},
\cite{bp}, \cite{BPS} .  Recall that for a graded-polarized
integral mixed Hodge structure $V=(V_\Z,W,F, q_k)$ the set of integral
weight zero Hodge classes is $\Hdg_0(V)_\Z:=
\Hom_{\bZ-\mathrm{MHS}}(\bZ(0), V) =(W_0)_\Z\cap F^0$, and we define
$\Hdg_0^d(V)_\Z\subset \Hdg_0(V)_\Z$ as the subset of Hodge classes
$v$ with $q_0(v,v)\leq d$, where $q_0$ is the polarization form on
$\Gr^W_0V_\Z$.  The locus $\Hdg^d_0(\Gamma\backslash
\cM)\subset\Gamma\backslash \cM$ of points $V$ for which
$\Hdg_0^d(V)\neq 0$ is a definable analytic subspace, and for any
period map $\phi:S\to\Gamma\backslash\cM$ we define
$\Hdg_0^d(S)\subset S$ to be the pullback of
$\Hdg^d_0(\Gamma\backslash \cM)$ with its natural
not-necessarily-reduced structure as a definable analytic subspace. 

\begin{cor}  Let $\phi:S\to\Gamma\backslash \cM$ be as in the theorem.  Then the Hodge subspace $\Hdg_0^d(S)\subset S$ is algebraic.
\end{cor}

Without too much difficulty, the same can be shown for the locus of bounded-norm Hodge classes $\Hdg_0^d(\cV)\subset\cV_\C$ in the total space of an admissible variation $(\cV,\cW,\cF)$, but we leave this to the reader.

\subsection{Outline}

In \S2 we recall some facts about how definable quotients work, and the relation between definable structures on quotients and choices of fundamental sets.
In \S3 we recall relevant background from mixed Hodge theory and the various real splittings that we use, and put a definable structure on graded-polarized integral mixed period spaces.
In \S4 we give a notion of variations of mixed Hodge structure and period maps on arbitrary algebraic varieties, review the notion of admissibility and its consequences, and state our main theorem.
In \S5 we prove our main theorem. Finally, in \S6 we generalize the construction of \S3 to place a definable structure on mixed Hodge varieties and prove functoriality.

\section{Definable quotients}
In this section we fix an o-minimal structure and we work in the category of definable locally compact Hausdorff topological spaces and definable continuous maps. \\

Let $X$ be a locally compact Hausdorff definable topological space and $\Gamma$ a group acting on $X$ by definable homeomorphisms.

\begin{defn}
A fundamental set for the action of $\Gamma$ on $X$ is an open definable subset $F \subseteq X$ such that
\begin{enumerate}
\item $\Gamma \cdot F = X$,
\item the set $\{ \gamma \in \Gamma \, | \, \gamma \cdot F  \cap F \neq \varnothing \}$ is finite.
\end{enumerate}
\end{defn}

\begin{remark}
The existence of a fundamental set for the action of $\Gamma$ on $X$ implies that $\Gamma$ equipped with the discrete topology acts properly on $X$. In particular, the set $\Gamma \backslash X$ equipped with the quotient topology is a locally compact Hausdorff topological space.
\end{remark}

\begin{prop}\label{prop fun sets}
If $F$ is a fundamental set for the action of $\Gamma$ on $X$, then there exists a unique definable structure on $\Gamma \backslash X$ such that the canonical map $F \arrow \Gamma \backslash X$ is definable.
\end{prop}

\begin{proof}
Let $R \subset X \times X$ be the equivalence relation associated to the action of $\Gamma $ on $X$, that is, $R = \{ (x, \gamma \cdot x) \}$. As observed before, the action of $\Gamma$ on $X$ is necessarily proper, hence $R$ is a closed subset of $X \times X$. It follows that the induced equivalence relation $R_F :=  R \cap (F \times F)$ on $F$ is closed. Moreover, the set
\[ R_F = \bigcup \limits_{\gamma \in \Gamma}  \{ (x,  \gamma \cdot x) \in F \times F \}  \]
is definable, since only matter the finite number of $\gamma$ for which $\gamma \cdot F  \cap F \neq \varnothing$. We conclude using that $ \Gamma \backslash X = F / R_F$ as topological spaces and that the equivalence relation $R_F$ on $F$ is closed, definable and \'etale.
 \end{proof}

\begin{remark}
\begin{enumerate}
\item The definable structure on $\Gamma \backslash X$ depends on the choice of $F$. (For example, two strips in $\bC$ with different slopes give different definable structures on the quotient $\bC / \bZ = \bC^\ast$.)
\item Two fundamental sets $F$ and $F^\prime$ define the same
definable structure on $\Gamma \backslash X$ if and only if
  $F$ is contained in a finite union of translates of $F^\prime$ under
  elements of $\Gamma$.
\item The map $F \arrow \Gamma \backslash X$ admits locally on the base some continuous definable section. Therefore, giving a morphism from a definable space $Y$ to $\Gamma \backslash X$ is equivalent to
giving a finite definable cover $Y = \cup Y_i$ and morphisms $Y_i \arrow F$ such that the induced maps $Y_i \arrow \Gamma \backslash X$ coincide on overlaps.
\item In case $X$ is a complex manifold and $\Gamma$ acts by definable
  biholomorphism the construction above is compatible with the complex structure.
\end{enumerate}
\end{remark}

\begin{prop}
If $F$ is a fundamental set for the action of $\Gamma$ on $X$ and $\Gamma^\prime \subset \Gamma$ is a finite index subgroup, then for any finite subset $C \subset \Gamma$ mapping surjectively onto $\Gamma / \Gamma^\prime$ the set $\bigcup_{\gamma\in C} \, \gamma \cdot F $ is a fundamental set for the action of $\Gamma^\prime$ on $X$. Moreover, the induced definable structure on $\Gamma^\prime \backslash X$ is independent of $C$ and the map $\Gamma^\prime \backslash X \arrow \Gamma \backslash X$ is definable.
\end{prop}

\begin{prop}\label{definability by pull-back}
Let $X$ and $Y$ be locally compact Hausdorff definable topological spaces and $\Gamma$ a group acting on both $X$ and $Y$ by definable homeomorphisms. Let $f : X \arrow Y$ be a $\Gamma$-equivariant continuous map.
\begin{itemize}
\item If $F$ is a fundamental set for the action of $\Gamma$ on $Y$, then $f^{-1}(F)$ is a fundamental set for the action of $\Gamma$ on $X$.
\item The induced continuous map $\Gamma \backslash X \arrow \Gamma \backslash Y$ is definable.
\end{itemize}
\end{prop}
\begin{proof}
First note that $\Gamma \cdot f^{-1}(F) = f^{-1}(\Gamma \cdot F) = f^{-1}(Y) = X$. On the other hand, the set $\{ \gamma \in \Gamma \, | \, \gamma \cdot f^{-1}(F)  \cap f^{-1}(F) \neq \varnothing \}$ is finite, since it is clearly a subset of $\{ \gamma \in \Gamma \, | \, \gamma \cdot F  \cap F \neq \varnothing \}$, and the latter is finite by assumption.
\end{proof}

\subsection{Fundamental sets for arithmetic groups}  Arithmetic quotients of reductive groups are endowed with $\bR_{\alg}$-definable structures using a Siegel set fundamental domain:

\begin{thm}[Theorem 1.1 of \cite{bkt}]\label{definable structure on arithmetic varieties}
Let $\bfG$ be a reductive algebraic group over $\bQ$, $\Gamma \subset \bfG(\bQ)$ an arithmetic subgroup and $M \subset \bfG(\bR)$ a compact subgroup. 
Then the quotient $\Gamma \backslash \bfG(\bR)/M $ admits a structure of $\bR_{\alg}$-definable analytic space, functorial in the triple $( \bfG, \Gamma, M)$ and characterized by the following property.
Let $\bfG(\bR) / M $ be endowed with its natural semi-algebraic structure and $\mathfrak{S} \subset \bfG(\bR) / M$ be an open semi-algebraic Siegel set.  Then $\mathfrak{S} \arrow  \Gamma \backslash \bfG(\bR)/M $ is $\bR_{\alg}$-definable.
\end{thm}

Note that the statement in \cite{bkt} is for $\bfG$ semi-simple (and that is all we will need), although the reductive case easily follows.

\section{Background in mixed Hodge theory}

\subsection{Splittings}\label{splittings}
(cf. \cite[\S2]{cks}.)\\

Fix a field $K$ of characteristic zero and a finite-dimensional $K$-vector space $V$ equipped with an increasing filtration $\{W_k\}$. Note that any $K$-vector space obtained from $V$ using duals, tensor products and subspaces inherits an increasing filtration from $\{W_k\}$.

\begin{defn}
A splitting of $\{W_k\}$ is a direct sum decomposition $V = \bigoplus \limits_k V_k$ such that $W_l = \bigoplus \limits_{k \leq l} V_k$. 
\end{defn}

Let $\cS(W)$ denote the variety of all splittings of $\{W_k\}$. It is a smooth algebraic variety defined over $K$ (a Zariski-open in a product of Grassmanians) such that $\cS(W)(L)$ is the set of all splittings of $\{W_k \otimes_K L\}$ for every field $L \supseteq K$. \\

The natural left action of $\GL(V)$ on $V$ induces an algebraic left action of the $K$-algebraic group $\GL(V)^W = \{ g \in \GL(V) \, | \, g(W_k) \subseteq W_k \text{ for all } k \} $ on $\cS(W)$. Its unipotent radical is the $K$-algebraic subgroup $U := \exp (W_{-1} \End(V))$. One easily checks that for every field $L \supseteq K$ the group $U(L)$ acts simply transitively on $\cS(W)(L)$, cf. \cite[\S3.6]{Kostant} or \cite[Proposition 2.2]{cks}.\\

There is a natural closed immersion $\cS(W) \hookrightarrow W_0 \End(V)$ which on $\bar K$-points associates to any given splitting $V = \bigoplus_k V_k$ the semisimple endomorphism $T \in \End(V)$ with integral eigenvalues whose $l$-eigenspace is $V_l$. This realizes $\cS(W)$ as an affine subspace of $W_0 \End(V)$ directed by the subvector space $W_{-1} \End(V)$.  In this realization, the left action of $\GL(V)^W$ on $\cS(W)$ is induced by the adjoint action of $\GL(V)$ on $\End(V)$, and the $K$-algebraic group $U$ acts on $\cS(W)$ by affine transformations.\\

There is an exact sequence of $K$-algebraic groups
\[  1 \arrow U \arrow \GL(V)^W \arrow \GL(\Gr^W V) \arrow 1, \]
and the choice of a splitting $T \in \cS(W)(K)$ induces a section
$\GL(\Gr^W V) \arrow \GL(V)^W $, whose image we denote by $\GL(V)^T$. \\

\begin{notn}
In the sequel we will frequently identify a splitting of $W$ with the
corresponding semisimple endomorphism $T$ of $V$. 
\end{notn}

\subsection{Mixed Hodge structures}
(cf. \cite{Deligne-Hodge2}.)\\

A decreasing filtration $F$ of an object $V$ is said to be finite if there exists two integers $m$ and $n$ such that $F^m V = V$ and $F^n V = 0$, and similarly for increasing filtrations.
In what follows, all filtrations are implicitely supposed to be finite.\\

Let $R = \bZ, \bQ$ or $\bR$. A mixed $R$-Hodge structure is a triple $V = (V_R, W, F)$ consisting of
\begin{itemize}
\item a $R$-module $V_R$ of finite type,
\item an increasing filtration $\{W_k\}$ of $V_R$ (the weight filtration),
\item a decreasing filtration $\{F^p\}$ of $V_{\bC} := V_R \otimes_R \bC$ (the Hodge filtration),
\end{itemize}
such that $\Gr_F^p \Gr_{\bar F}^q \Gr_l^{W} \left( V_{\bC} \right) = \{0\}$ when $p + q \neq l$, where we denote by the same symbol $W$ the filtration induced by $W$ on $V_{\bC}$ and by $\bar F$ the conjugate filtration of $F$ defined by $\bar F^q := \overline{F^q}$. \\

A morphism of mixed $R$-Hodge structure is a $R$-linear morphism of the underlying $R$-modules that preserves both filtrations. The category of mixed $R$-Hodge structures is abelian, and it admits both duals and tensor products (hence internal homs).\\

Let $V = (V_R, W, F)$ be a pure $R$-Hodge structure of weight $n$, meaning that $ \Gr_l^{W} \left( V_R \right) = \{0\}$ when $l \neq n$. In that case, we have the Hodge decomposition
\[  V_{\bC} = \bigoplus_{p + q = n} V^{p,q}  \]
with $V^{p,q} := F^p \cap \bar F^q$, so that $V^{q,p} = \overline{ V^{p,q}}$. The Weil operator $C \in \End(V_{\bR})$ is then the real endomorphism satisfying
\[ C_ {\bC} = \bigoplus_{p,q} i^{p-q} \cdot \id_{V^{p,q}}.\]

Let $q : V_R \otimes V_R \rightarrow R$ be a $(-1)^n$-symmetric bilinear form---that is, $q$ is symmetric if $n$ is even and skew-symmetric if $n$ is odd. We say that the pure $R$-Hodge structure $V$ is polarized by $q$ if the hermitian form $h$ on $V_{\bC}$ defined by $h(u, v) = q_{\bC} (Cu, \bar v)$ is positive-definite and the Hodge decomposition of $V_{\bC}$ is $h$-orthogonal.

\subsection{Bigradings}\label{bigradings}
\begin{defn}
A bigrading of a real mixed Hodge structure $(V, W,F )$ is a direct sum decomposition $V_{\bC} = \bigoplus \limits_{p,q} J^{p,q}$ such that:
\begin{align*}
F^p = \bigoplus \limits_{ r \geq p,s} J^{r,s} \text{ and } (W_k)_{\bC} = \bigoplus \limits_{r+s \leq k} J^{r,s}.
\end{align*}
\end{defn}

The bigradings of a real mixed Hodge structure $(V,W,F)$ are easily seen to be in bijection with the splittings $T \in \cS(W_{\bC})$ such that $T( F^p) \subseteq F^p$, via $V_l(T) = \bigoplus \limits_{p+q = l} J^{p,q}$.

\begin{lem}[Deligne \cite{Deligne-Hodge2}]\label{Deligne bigrading}
If $(V, W,F )$ is a real mixed Hodge structure, then it admits a unique bigrading $\{I^{p,q} \}$ which satisfies:
\begin{equation*}
 I^{p,q} = \bar{I^{q,p}}  \mod  \bigoplus \limits_{r<p ,s<q} I^{r,s}  .
\end{equation*}
\end{lem}

Deligne bigrading is functorial and is given explicitely by the formula:
\[ I^{p,q} := \left (F^p \cap (W_{p+q})_{\bC} \right) \cap \left( \bar F^q \cap (W_{p+q})_{\bC} + \sum \limits_{j \geq 0} \bar F^{q-1-j}\cap (W_{p+q-2-j})_{\bC} \right). \]

\subsection{Real splittings}

\begin{prop}
A real mixed Hodge structure $(V, W, F)$ is said to split over $\bR$ if it satisfies one of the equivalent following properties:
\begin{enumerate}
\item it is a direct sum of pure real Hodge structures of different weights,
\item it admits a real splitting, i.e. a bigrading $\{J^{p,q} \}$ such that $J^{p,q} = \bar J^{q,p}$,
\item there exists $T \in \cS_ {\bR}(W)$ such that $T(F^p) \subseteq F^p$.
\end{enumerate}
\end{prop}

 If $(V,W, F)$ admits a real splitting $\{J^{p,q} \}$, then one has necessarily 
 \[ J^{p,q} = F^p \cap \bar F^q \cap W_{p+q}, \]
 so that it is unique and coincides with Deligne bigrading.\\

\subsection{Graded-polarized mixed period domains}\label{sect period domains}
(cf. \cite{usui, pearlstein}.)\\

Let $V$ be a finite dimensional $\bR$-vector space equipped with an increasing filtration $\{W_k\}$ and a collection of non-degenerate bilinear forms $q_k : \Gr_k^W V \otimes_{\bR} Gr_k^W V \arrow \bR$ that are $(-1)^k$-symmetric. Fix a partition of $\dim_{\bR} V$ into non-negative integers $\{h^{p,q} \}$ such that $h^{p,q} = h^{q,p}$.\\

For any integer $k$, we denote by $\Omega_k$ the Griffiths period domain parametrizing all decreasing filtrations $\{ F_k^p \}_p$ on $\Gr_k^W V_{\bC}$ with $\dim_{\bC} F_k^p = \sum_{r \geq p} h^{r, k-r} $ that define a real pure Hodge structure of weight $k$ polarized by $q_k$, and by $\check \Omega_k$ its compact dual parametrizing the $(q_k)_{\bC}$-isotropic filtrations $\{ F_k^p \}_p$ on $\Gr_k^W V_{\bC}$ with $\dim_{\bC} F_k^p = \sum_{r \geq p} h^{r, k-r} $. Letting $\check \Omega := \prod_k \check \Omega_k$ and $ \Omega := \prod_k \Omega_k$, and denoting by $\bfH$ the real algebraic group $ \prod_k   \Aut(q_k)$, it follows from Griffiths theory that $\check \Omega $ is a smooth projective complex variety on which the complex algebraic group $ \bfH(\bC)$ acts transitively by algebraic automorphisms and $\Omega \subset \check \Omega$ is a real semi-algebraic open subset on which the real algebraic group $ \bfH(\bR)$ acts transitively by semi-algebraic automorphisms.\\

Let $\cM$ denote the corresponding mixed period domain, i.e. the set of decreasing filtrations $\{ F^p \}$ of $V_{\bC}$ such that $(V, W, F)$ is a real mixed Hodge structure graded-polarized by the $q_k$'s and such that 
\[ \dim_{\bC} \left((F^p \Gr_{p + q}^W V_{\bC}) \cap (\bar F^q \Gr_{p + q}^W V_{\bC}) \right)=  h^{p,q}. \]

By definition, $\cM$ is a semi-algebraic open subset of the smooth projective complex variety $\check \cM$ that parametrizes the decreasing filtrations $\{F^p\}$ of $V_{\bC}$ by complex vector subspaces such that the filtration induced on the graded pieces $\Gr_k^WV_\C$ is inside $\check\Omega_k$ for each $k$.

Let $\bfG$ denotes the real algebraic group defined as the preimage of $\bfH \subset \bfGL( \Gr^W V)$ through the natural homomorphism $\bfGL(V)^W \arrow \bfGL( \Gr^W V)$. Let $G$ denote the preimage of $\bfH(\bR)$ by the homomorphism $\bfG(\bC) \arrow \bfH(\bC)$. It is naturally a group object in the category of $\bR_{\alg}$-definable topological spaces, and the following inclusions hold in this category:
\[  \bfG(\bR) \subseteq G = \bfU(\bC) \cdot\bfG(\bR) \subseteq \bfG(\bC) .\]
Moreover the action of $\bfG(\bC)$ on $\check \cM$ induces an action
of $G$ on $\cM$. The following proposition is well-known, see
  \cite[Prop. 2.11]{usui} for instance:

\begin{prop}
 The real algebraic group $G$ acts transitively on $\cM$ by semi-algebraic automorphisms.
\end{prop}

\begin{proof}
Recall that $\cS(W)$ denotes the variety of splittings of $W$,
cf. section \ref{splittings}. The complex variety $\Omega \times
\cS(W)(\bC)$ parametrizes the elements of $\cM$ equipped with a
bigrading, cf. section \ref{bigradings}. Thanks to the existence of
Deligne bigrading, the natural map $ \Omega \times \cS(W)(\bC) \arrow
\cM$ is surjective. Since this map is also $G$-equivariant and the
$G$-action on $ \Omega \times \cS(W)(\bC) $ is 
transitive by \cite[Prop. 2.2]{cks}, it follows that $G$ acts transitively on $\cM$.
\end{proof}

Note that the morphism $ \cM \arrow  \Omega$ which is equivariant with respect to the homomorphism $G \arrow \bfH(\bR)$
is the restriction of a complex algebraic map $\check \cM \arrow \check \Omega$ which is equivariant with respect to the homomorphism $\bfG(\bC) \arrow \bfH(\bC)$.\\
 
Let $\cM_ {\bR} \subset \cM$ denote the subset consisting of those Hodge filtrations for which the corresponding mixed Hodge structure is split over $\bR$. The group $\bfG(\bR)$ acts transitively on $\cM_ {\bR}$, so that it is a smooth real semi-algebraic subset  of $\cM$. Moreover, $\cM_{\bR}$ is naturally in bijection with $\Omega \times \cS(W)(\bR)$, and this bijection is compatible with the $\bfG(\bR)$-actions, so that it is an isomorphism of real semi-algebraic spaces.\\

Observe that the action of $G$ on $\cM$ is not proper, since the stabilizer of a point is non-compact.
\begin{prop}
The actions of $\bfG(\bR)$ on $\cM_ {\bR}$ and $ \cM$ are proper.
\end{prop}
\begin{proof}
Let $\cB_{\bR}$ denote the set of real Hodge frames of mixed Hodge structures that are split over $\bR$. It is a $\bfG(\bR)$-torsor, hence the $\bfG(\bR)$-action on $\cB_{\bR}$ is proper. But the surjective and proper morphism $\cB_{\bR} \arrow \cM_{\bR}$ is $\bfG(\bR)$-equivariant, therefore the $\bfG(\bR)$-action on $\cM_{\bR}$ is proper too, cf. \cite[proposition 5.i) in TG III.29]{Bourbaki}. By \cite[proposition 5.ii) in TG III.29]{Bourbaki}, the properness of the action of $\bfG(\bR)$ on $\cM$ follows, once we know the existence of a $\bfG(\bR)$-equivariant continuous map $\cM \arrow \cM_{\bR}$, for which we can refer for example to proposition \ref{delta-splitting}.
\end{proof}

\begin{cor}
If $\Gamma$ is a discrete subgroup of $\bfG(\bR)$, then the induced action of $\Gamma$ on $\cM$ is proper and the quotient $\Gamma \backslash \cM$ admits a canonical structure of complex analytic space such that the natural map $\cM \arrow \Gamma \backslash \cM$
is holomorphic.
\end{cor}

 \subsection{The $\delta$-splitting}
 
 Given a real mixed Hodge structure $(V,W,F)$ with Deligne bigrading $\{I^{p,q} \}$, we define a nilpotent Lie subalgebra of $\End(V)_{\bC}$ by
 \[ L^{-1, -1}_ {(W,F)} = \{ X \in \End(V_ {\bC}) \, | \, X(I^{p,q}) \subseteq \bigoplus \limits_{r < p, s < q} I^{r,s}  \} . \] 
 
 It is defined over $\bR$ with real form $(L^{-1, -1}_{(W,F)})_\bR := L^{-1, -1}_ {(W,F)} \cap \End(V)$.
 
 \begin{prop}[Deligne, cf. proposition 2.20 in \cite{cks}]
Given a real mixed Hodge structure $(V,W, F)$, there exists a unique $\delta \in (L^{-1, -1}_ {(W,F)})_\bR$ such that $(V, W, e^{-i \delta} \cdot F)$ is a real mixed Hodge structure
which splits over $\bR$.
\end{prop}

This splitting is functorial ($\delta$ commutes with every morphism of real mixed Hodge structures) and satisfies $L^{-1, -1}_ {(W,F)} = L^{-1, -1}_ {(W, e^{-i \delta} \cdot F)} $.

\begin{prop}[{\cite[proposition 2.24]{cks}}]\label{delta-splitting}
The Deligne $\delta$-splitting yields a $\bfG(\bR)$-equivariant smooth real semi-algebraic retraction $\cM \arrow \cM_ {\bR}$ of the inclusion $\cM_ {\bR} \subseteq \cM$ (over $\Omega$).
\end{prop}

 \subsection{The $\sl_2$-splitting, aka canonical splitting, aka $\xi$-splitting}

\begin{thm}[Deligne, cf. {\cite[Theorem 2.18]{bp}}]
The $\sl_2$-splitting is the unique, functorial splitting of real mixed Hodge structures which is given by universal Lie polynomials in the Hodge components of the Deligne $\delta$-splitting such that if $(\exp(z N) \cdot F, W)$ is an admissible nilpotent orbit with limit mixed Hodge structure $(F, M)$ which is split over $\bR$ then the Deligne grading of the splitting of $(\exp i N \cdot F, W)$ is a morphism of type $(0,0)$ for $(F, M)$.
\end{thm}

\begin{cor}
The $\sl_2$-splitting yields a $\bfG(\bR)$-equivariant smooth real semi-algebraic retraction $r:\cM \arrow \cM_ {\bR}$ of the inclusion $\cM_ {\bR} \subseteq \cM$ (over $\Omega$).
\end{cor}

 \subsection{The definable structure on arithmetic quotients of period domains}\label{definable structure on mixed period domains}
 In the following proposition we continue to identify $\cM_\bR=\Omega\times \cS(W)(\bR)$.

 \begin{prop}\label{prop def struct real split}  Let $\Gamma\subset \bfG(\bQ)$ be an arithmetic subgroup.  Then $\Gamma\backslash \cM_\bR$ admits a structure of a $\bR_{\alg}$-definable analytic space characterized by the following property:  for any semi-algebraic Siegel set $\mathfrak{S}\subset\Omega$ and bounded semi-algebraic $\Sigma\subset \cS(W)(\bR) $, the map $\mathfrak{S}\times\Sigma\to \Gamma\backslash\cM_\bR$ is $\bR_{\alg}$-definable.
\end{prop}
\begin{proof}
Let $\bfU$ be the unipotent radical of $\bfG$, $\Gamma_\bfU:=\Gamma\cap \bfU(\Q)$, and $\Gamma_\bfH$ the image of $\Gamma$ in $\bfH(\Q)$.  By \cite[Prop. 2.2]{cks}, $\bfU(\bR)$ acts simply transitively on $\cS(W)(\bR)$.  Taking $B\subset \cS(W)(\bR)$ to be a bounded semi-algebraic fundamental set for the cocompact action of $\Gamma_\bfU$ and $F$ to be a definable fundamental set for $\Gamma_\bfH\backslash\Omega$, we use $F\times B$ as a definable fundamental set to induce the definable structure on $\Gamma\backslash\cM_\bR$ via Proposition \ref{prop fun sets}. 
Let $\mathfrak{S}\subset\Omega$ be any semi-algebraic Siegel set and $\Sigma\subset \cS(W)(\bR) $ be any bounded semi-algebraic subset. Then $\mathfrak{S}$ meets only finitely many $\Gamma_\bfH$-translates of $F$, and for any any $\gamma\in \bfG(\Q)$ we have that $\Sigma$ meets only finitely many $\Gamma_\bfU$-translates of $\gamma B$, so $\mathfrak{S}\times \Sigma$ meets only finitely many $\Gamma$-translates of $F\times B$. Therefore the map $\mathfrak{S}\times\Sigma\to \Gamma\backslash\cM_\bR$ is $\bR_{\alg}$-definable.
\end{proof}

\begin{defn}\label{defn def structure} Let $\Gamma\subset \bfG(\bQ)$ be an arithmetic subgroup and $r:\cM\to\cM_\bR$ the $\sl_2$-splitting.  Let $\Xi\subset\cM_\bR$ be a definable fundamental set for $\Gamma\backslash \cM_\bR$.  We endow $\Gamma\backslash \cM$ with a structure of a $\bR_{\alg}$-definable analytic space via proposition \ref{definability by pull-back} using $r^{-1}(\Xi)$ as a definable fundamental set. 
\end{defn}

Be careful that two different retractions will yield in general two different definable structures on $\Gamma \backslash \cM$.

\section{Variations of mixed Hodge structures and their period maps}

\subsection{Variation of mixed Hodge structures}
Let $R= \bZ, \bQ$ or $\bR$. A variation of mixed $R$-Hodge structures over a (possibly non-reduced) complex analytic space $S$ is the data of

\begin{itemize}
\item a $R$-local system $\cL$ on the underlying topological space,

\item an increasing filtration $\cW$ of $\cL$ by sublocal systems (the weight filtration),

\item a decreasing filtration $F$ of $\cL \otimes_{R} \cO_S$ by locally split $\cO_S$-submodules (the Hodge filtration)

\end{itemize}

such that 

\begin{itemize}
\item $F$ satisfies Griffiths transversality in the usual sense on the reduced\footnote{Note in particular that we do not require the nilpotent tangent directions to be Griffiths transverse, though it is not clear that this level of generality is useful:  variations coming from geometry will satisfy Griffiths transversality in the nilpotent directions as well.} regular locus of $S$.

\item for every $s \in S$, $(\cL_s, \cW_s, F_s)$ is a mixed $R$-Hodge structure.
\end{itemize}
We say the variation is graded-polarized if we are given a parallel polarization on each of the associated variation of pure Hodge structures. 

\begin{lemma}\label{unipotent monodromies}
Consider a variation of integral mixed Hodge structures over a complex analytic space $S$. Then, up to replacing $S$ by a finite \' etale cover, the pull-back of the underlying local system by any holomorphic map $\Delta^\ast \arrow S$ has unipotent monodromy.
\end{lemma}
\begin{proof} By going to a finite \'etale cover, one can assume that all the monodromy operators are trivial modulo a prime number $p \geq 3$. On the other hand, by applying Borel's monodromy theorem \cite[Lemma 4.5]{Schmid1} to the associated variations of pure Hodge structures, one sees that the eigenvalues of the monodromy operator corresponding to a holomorphic map $\Delta^\ast \arrow S$ are roots of the unity. Since roots of unity of a fixed degree inject modulo $p$ for sufficiently large $p$, the claim follows.
\end{proof}

\subsection{Period maps}Let $S$ be a (possibly non-reduced) complex analytic space.  By a mixed period map from $S$ we mean a locally liftable analytic map $\phi : S \arrow \bfG(\bZ) \backslash \cM$ which is tangent to the Griffiths transverse foliation of $\cM$ on the reduced regular locus of $S$.  Evidently, a mixed period map from $S$ is equivalent to giving a variation of graded-polarized integral mixed Hodge structures on $S$ in the sense of the previous section.

\subsection{Admissibility}

The notion of admissibility for a variation of mixed Hodge structures was introduced by Steenbrink and Zucker over one-dimensional bases \cite{sz} and by Kashiwara \cite{kashiwara} in higher dimensions. Let us recall the definitions.\\

Let $(\cL, \cW, F)$ be a graded-polarizable variation of real mixed Hodge structures on $\Delta^\ast$ with unipotent monodromy. Let $\bar \cV $ and $\bar \cW_k$ denote the canonical extensions of $\cL \otimes_{\bR} \cO_{\Delta^\ast}$ and $\cW_k \otimes_{\bR} \cO_{\Delta^\ast}$ to $\Delta$ respectively, equipped with their logarithmic connections. The variation $(\cL, \cW, F)$ is called pre-admissible if the following conditions hold:
\begin{enumerate}
\item The residue at the origin of the logarithmic connection on $\bar \cV$, which is an endomorphism of the fiber $\bar \cV_{|0}$ of $\bar \cV$ at the origin, admits a weight filtration relative to $\bar \cW_{|0}$.
\item The Hodge filtration $F$ extends to a subbundle $\bar F$ of $\bar \cV $ such that $\Gr_ {\bar F}^p  \Gr_k^{\bar \cW} \bar \cV$ is locally-free for all $p$ and $k$.
\end{enumerate}

Given a Zariski-open subset $S$ in a reduced complex analytic space $\bar S$, we say that a graded-polarized variation of real mixed Hodge structures $(\cL, \cW, F)$ on $S$ is admissible with respect to the inclusion $S \subset \bar S$ if for any holomorphic map $f : \Delta \arrow \bar S$ such that $f(\Delta^\ast) \subset S$ and $f^\ast \cL$ has unipotent monodromy, the pull-back variation on $\Delta^\ast$ is pre-admissible. \\

One easily verifies that a variation of real mixed Hodge structures on $\Delta^\ast$ with unipotent monodromy which is pre-admissible is admissible with respect to the inclusion $\Delta^\ast \subset \Delta$, cf. \cite[lemma 1.9.1]{kashiwara}.

\begin{prop}
If a graded-polarizable variation of real mixed Hodge structures over a complex algebraic variety $S$ is admissible with respect to an algebraic compactification $\bar S$ of $S$, then it is admissible with respect to any other algebraic compactification of $S$.
\end{prop}
\begin{proof}
Indeed, a holomorphic map $f : \Delta^\ast \arrow S$ is the restriction of a holomorphic $ \Delta \arrow \bar S$ exactly when it is definable in $\bR_{\an}$, hence this property in independent of the compactification.
\end{proof}

\subsection{Nilpotent orbit theorem}

Consider a graded-polarized variation of real mixed Hodge structures over $ (\Delta^\ast)^n$ with unipotent monodromies. Let $\bH$ denote the Poincar\'e upper half-plane and $e : \bH^n \arrow (\Delta^\ast)^n$ the uniformizing map given by $e(z_1, \cdots, z_n) = (\exp(2 \pi i \cdot z_1), \cdots , \exp(2 \pi i \cdot z_n))$.  Choosing a reference point in $\bH^n$, we get a period map $\tilde \phi : \bH^n \arrow \cM$. Denoting by $N_j$ ($1 \leq j\leq n$) the logarithm of the monodromy operators corresponding to counterclockwise simple circuits around the various punctures, the holomorphic map $\tilde \Psi : \bH^n \arrow \check \cM$ given by $\tilde \Psi(z) := \exp( - \sum \limits_{j=1}^n z_j \cdot N_j) \cdot \tilde \phi(z)$ factorizes through the projection map $e : \bH^n \arrow  (\Delta^\ast)^n$. Let $ \Psi : (\Delta^\ast)^n \arrow \check \cM$ denote the factorization. Thanks to Schmid's nilpotent orbit theorem \cite[Theorem 4.12]{Schmid1}, the composition of $ \Psi$ with the projection $\check \cM \arrow \check \Omega$ extend as a holomorphic map $\Delta^n \arrow \check \Omega$. If one assume from now on that the variation is admissible with respect to the inclusion $ (\Delta^\ast)^n \subset \Delta^n$, then by definition the restriction of $\Psi$ to any punctured disk $\Delta^\ast \subset (\Delta^\ast)^n$ extends as a holomorphic map $\Delta \arrow \check \cM$. Since the projection $\check \cM \arrow \check \Omega$ is an affine holomorphic map, it follows that $\Psi $ extends as a holomorphic map $\Delta^n \arrow \check \cM$. Therefore we have proved:
\begin{prop}\label{nilpotent orbit}
Let $S$ be the complement of a normal crossing divisor in a complex manifold $\bar S$. Let $(\cL, \cW, F)$ be a graded-polarized variation of real mixed Hodge structures over $S$ with unipotent monodromies at infinity which is admissible with respect to the inclusion $S \subseteq \bar S$. If $\bar \cV $ and $\bar \cW_k$ denote the canonical extensions of $\cL \otimes_{\bR} \cO_{\Delta^\ast}$ and $\cW_k \otimes_{\bR} \cO_{\Delta^\ast}$ to $\bar S$ respectively, then the Hodge filtration $F$ extends to a subbundle $\bar F$ of $\bar \cV $ such that $\Gr_ {\bar F}^p  \Gr_k^{\bar \cW} \bar \cV$ is locally-free for all $p$ and $k$.
\end{prop}

 \subsection{Admissible period maps are definable}

\begin{thm}\label{definability of period maps}
Consider an admissible variation of graded-polarized integral mixed Hodge structures over a reduced complex algebraic variety $S$, and let $\phi : S \arrow \bfG(\bZ) \backslash \cM$ be the associated period map.
Then $\phi$ is definable in $\bR_{\an, \exp}$, where we equip $ \bfG(\bZ) \backslash \cM$ with the $\bR_{\alg}$-definable structure associated to the $\sl_2$-splitting, cf. section \ref{definable structure on mixed period domains}.
\end{thm}

This generalizes to the mixed case \cite[Theor.1.3]{bkt} for
  pure variations of Hodge structures.
\section{Proof of Theorem \ref{definability of period maps}}

Recall that it is sufficient to prove the definability of the map obtained by precomposing $\phi$ by a surjective definable holomorphic map. In particular, by looking at a desingularization of $S$, one can assume from the beginning that $S$ is smooth. Moreover, up to replacing $S$ by a finite \' etale cover, one can assume that the monodromies at infinity are unipotent, c.f. lemma \ref{unipotent monodromies}.

Taking a covering of $S$ in $\bR_{\alg}$ (or just $\bR_{\an, \exp}$) by open subsets isomorphic to $(\Delta^\ast)^n$, one sees that we are reduced to proving:

\begin{thm}
Consider an admissible variation of graded-polarized integral mixed Hodge structures with unipotent monodromies over the  punctured  polydisk $(\Delta^\ast)^n$, and let $\phi : (\Delta^\ast)^n \arrow \bfG(\bZ) \backslash \cM$ be the associated period map. Then $\phi$ is definable in $\bR_{\an, \exp}$.
\end{thm}

Let $\bH$ denote the Poincar\'e upper half-plane and $e : \bH^n \arrow (\Delta^\ast)^n$ the uniformizing map given by $e(z_1, \cdots, z_n) = (\exp(2 \pi i \cdot z_1), \cdots , \exp(2 \pi i \cdot z_n))$. By choosing a lifting $\tilde \phi$ of the period map $\phi$, we obtain a commutative diagram of holomorphic maps
\[\begin{tikzcd}
 \bH^n  \arrow{r}{\tilde \phi}\arrow{d}[swap]{e}&\cM \arrow{d} \\
(\Delta^\ast)^n \arrow{r}[swap]{\phi} &  \bfG(\bZ) \backslash \cM 
\end{tikzcd}\]

A vertical strip in $\bH^n$ is by definition a product of sets of the form
\[ \{ (x,y) \in \bH \, | \,  a < x < b, \, c < y  \} \]
for some real numbers $a < b$ and $c>0$. Let $\mathfrak{S} \subset \bH^n$ be a vertical strip mapped by $e$ surjectively onto $(\Delta^\ast)^n$, and consider the induced commutative diagram of holomorphic maps
\[\begin{tikzcd}
\mathfrak{S}  \arrow{r}{\tilde \phi_{|\mathfrak{S}}}\arrow{d}[swap]{e_{|\mathfrak{S}}}&\cM \arrow{d} \\
(\Delta^\ast)^n \arrow{r}[swap]{\phi} &  \bfG(\bZ) \backslash \cM 
\end{tikzcd}\]

Since the holomorphic map $e_{|\mathfrak{S}}$ is definable and surjective, the definability of $\phi$ will be proved if we show that $\tilde \phi_{|\mathfrak{S}} : \mathfrak{S} \arrow \cM$ is definable and that the image of $\mathfrak{S}$ by $\tilde \phi$ is contained in a finite union of definable fundamental sets. This is the content of the next two results.

\begin{prop}
If $\tilde \phi : \bH^n \arrow \cM$ is a lifting of the period map of an admissible variation of mixed Hodge structures over $(\Delta^\ast)^n$ with unipotent monodromies, then its restriction to any vertical strip is definable in $\bR_{\an, \exp}$.
\end{prop}
\begin{proof}
Denoting by $N_j$ ($1 \leq j\leq n$) the logarithm of the monodromy operators corresponding to counterclockwise simple circuits around the various punctures, the holomorphic map $\tilde \Psi : \bH^n \arrow \check \cM$ given by $\tilde \Psi(z) := \exp( - \sum \limits_{j=1}^n z_j \cdot N_j) \cdot \tilde \phi(z)$ factorizes through the projection map $e : \bH^n \arrow  (\Delta^\ast)^n$. If $ \Psi : (\Delta^\ast)^n \arrow \check \cM$ denote the factorization, it follows from the admissibility condition that $\Psi $ extends as a holomorphic map $\Delta^n \arrow \check \cM$, cf. proposition \ref{nilpotent orbit}. For any vertical strip $\mathfrak{S} \subset \bH^n$, the restriction of $\Psi$ to its image by $e$ is the restriction to a relatively compact set of a holomorphic map, therefore it is definable in $\bR_{\an}$.  As $e_{ |\mathfrak{S}} : \mathfrak{S}  \arrow \Delta ^n $ is $\bR_{\an, \exp}$-definable, it follows that $(\mathfrak{S} \arrow \cM, z \mapsto \tilde \Psi(z) =\Psi(e(z))$ is $\bR_{an, exp}$-definable. Since both the action of $\bfG(\bC)$ on the compact dual $\check \cM$ and the morphism $\bC^n \arrow \bfG(\bC)$ given by $(z_1, \cdots, z_n) \mapsto \exp( \sum \limits_{j=1}^n z_j \cdot N_j) $ are algebraic, it follows from the equality $\tilde \phi(z) = \exp( \sum \limits_{j=1}^n z_j \cdot N_j) \cdot \tilde \Psi(z)$ that the restriction of $\tilde \phi$ to any vertical strip is definable in $\bR_{\an, \exp}$.
\end{proof}

\begin{prop}
If $\tilde \phi : \bH^n \arrow \cM$ is a lifting of the period map of an admissible variation of mixed Hodge structures over $(\Delta^\ast)^n$, then the image by $\tilde \phi$ of a vertical strip is contained in a finite union of definable fundamental sets.
\end{prop}

Given the definition of the definable structure in Definition
\ref{defn def structure}, the theorem is a consequence of its special pure case proved in \cite{bkt} and the following result of Brosnan--Pearlstein.

\begin{thm}[Cor 2.34 of \cite{bp}]

Let $\bH^r \arrow \cM$ be a lifting of the period map of an admissible variation of mixed Hodge structures over $(\Delta^\ast)^r$. If $\cM \arrow \cS(W)(\bR)$ is the map associated to the $\sl_2$-splitting, then the composition $\bH^r \arrow \cS(W)(\bR)$ is bounded on any vertical strip.

\end{thm}

\section{Mixed Hodge varieties}

\subsection{Mixed Mumford-Tate groups} (cf. \cite{andre} and \cite[\S 2]{kHodge}.)

We first briefly summarize Mumford--Tate groups of mixed Hodge structures.  For simplicity we focus on rational mixed Hodge structures, though the same holds for any subfield of $\bR$.  Let $\mathbf{S}=\Res_{\C/\bR}\bfG_m$, and define the weight torus to be the diagonal $w:\bfG_{m}\to\bfS$.  For a rational mixed Hodge structure $V=(V_\Q,W,F)$, the associated Deligne torus is the homomorphism $h:\mathbf{S}_\C\to \bfGL(V_\C)$ by which $(z_1,z_2)\in\mathbf{S}_\C(\C)=\C^*\times\C^*$ acts as $z_1^pz_2^q$ on $I^{p,q}$ in the Deligne splitting of $V$.  Recall that the weight zero Hodge classes of $V$ are defined as $\Hdg_0(V)_\Q:=(W_0)_\Q\cap F^0$. 

Let $\langle V\rangle$ be the smallest full subcategory of the category of rational mixed Hodge structures which contains both $V$ and $\Q(0)$ and is closed under subquotients, $\oplus$, and $\otimes$.  The Mumford--Tate group $\MT(V)\subset\bfGL(V)$ is then the Tannakian group associated to $\langle V\rangle$ with its obvious tensor functor.  By \cite[Lemma 2]{andre}, $\MT(V)$ is equal to the largest $\Q$-subgroup of $\bfGL(V)$ which fixes $\Hdg_0(T^{m,n}(V))$ for all $m,n\geq 0$ where $T^{m,n}(V):=V^m\otimes (V^\vee)^n$.  It is connected and equal to the $\Q$-Zariski closure of $h$ in $\bfGL(V)$, is contained in $\GL(V)^W$, and if $\Gr^W V$ is polarizable then $\MT(\Gr^WV)$ is the quotient of $\MT(V)$ by its unipotent radical (cf. \cite[\S 2.4]{kHodge}).

The Mumford--Tate group of an integral mixed Hodge structure is simply the Mumford--Tate group of the associated rational mixed Hodge structure.
\subsection{Mixed Hodge varieties} 
In this section we largely follow the setup in \cite[\S3]{kHodge}, which we refer to for details (see also \cite{pink}).  The following definition serves as an abstract model for a Mumford--Tate group.

\begin{defn}\label{def mixed datum}A mixed Hodge datum is a pair $(\bfG,X_\bfG)$ where $\bfG$ is a connected linear algebraic $\Q$-group and $X_\bfG$ is a $\bfG(\bR)\bfU(\C)$-conjugacy class of homomorphisms $\bfS_\C\to \bfG_\C$ where $\bfU$ is the unipotent radical of $\bfG$ satisfying the following conditions.  For some (hence any) $h\in X_\bfG$, letting $\bfH=\bfG/\bfU$, we have
\begin{enumerate}
\item $\bfS_\C\xrightarrow{h}\bfG_\C\to\bfH_\C$ is defined over $\bR$;
\item $\bfG_{m}\xrightarrow{w}\bfS_\C\xrightarrow{h}\bfG_\C\to\bfH_\C$ is defined over $\Q$;
\item The rational mixed Hodge structure on the Lie algebra $\frak{g}$ of $\bfG$ induced by the adjoint action has $W_{-1}\frak{g}=\frak{u}$.
\end{enumerate}
A morphism of mixed Hodge data $\rho: (\bfG,X_\bfG)\to(\bfG',X'_{\bfG'})$ is a $\Q$-homomorphism $\rho:\bfG\to\bfG'$ sending $X_\bfG$ to $X'_{\bfG'}$.
\end{defn}
The first two conditions guarantee that if $\rho :\bfG\to\bfGL(V_\Q)$ is a $\Q$-representation, then $\rho\circ h$ endows $V_\Q$ with the structure of a rational mixed Hodge structure for each $h\in X_\bfG$.  If $\rho$ is moreover faithful the third condition ensures that $\bfU$ is the group acting trivially on the associated graded.  When $\rho$ is faithful the map
\[X_\bfG\to\{\mbox{rational mixed Hodge structures on $V$}\}\]
factors through a complex manifold $\cD_{\bfG,X_\bfG}$ which is independent of $\rho$.

\begin{defn}\hspace{1in}
\begin{enumerate}
\item A connected mixed Hodge datum is a triple $(\bfG,X_\bfG,\cD^+)$ where $(\bfG,X_\bfG)$ is a mixed Hodge datum and $\cD^+$ is a connected component of $\cD_{\bfG,X_\bfG}$; the stabilizer $\bfG(\bR)^+$ of $\cD^+$ in $\bfG(\bR)$ is a connected component.  We refer to $\cD^+$ as a connected mixed Hodge domain.
\item For $(\bfG,X_\bfG,\cD^+)$ a connected mixed Hodge datum and $\Gamma\subset \bfG(\Q)^+:=\bfG(\Q)\cap\bfG(\bR)^+$ an arithmetic subgroup, the associated connected mixed Hodge variety is the complex manifold $\Gamma\backslash\cD^+$.
\item A morphism $f:\cD\to\cD'$ of connected mixed Hodge domains corresponding to connected mixed Hodge data $(\bfG^{(\prime)},X^{(\prime)}_{\bfG^{(\prime)}},\cD^{(\prime)})$ is a map induced from a $\Q$-homomorphism $\rho:\bfG\to\bfG'$ sending $X_\bfG$ to $X'_{\bfG'}$ and $\cD$ to $\cD'$.  If in addition $\Gamma$ is sent to $\Gamma'$ we call the induced map $\bar f:\Gamma\backslash \cD\to\Gamma'\backslash\cD'$ a morphism of connected mixed Hodge varieties.
\item A Hodge datum $(\bfG,X_\bfG)$ is graded-polarizable if for some (hence any) $h\in X_\bfG$ and some (hence any) faithful representation $\rho:\bfG\to\bfGL(V_\Q)$ the induced mixed Hodge structure on $V_\Q$ is graded-polarizable.  In this case we say the associated connected mixed Hodge domains and varieties are graded-polarizable as well.
\end{enumerate}
\end{defn}
\begin{remark}For simplicity we only deal with connected mixed Hodge varieties, as this is all that is needed for definability questions:  a general mixed Hodge variety as in \cite{kHodge} is a finite union of connected ones.
\end{remark}

Note that any connected mixed Hodge domain $\cD^+$ has a functorial $\bR_{\alg}$-definable structure for which the action of $\bfG(\bR)^+$ is definable.

For any graded-polarizable connected mixed Hodge datum $(\bfG,X_\bfG,\cD^+)$ and a faithful $\Q$-representation $\rho:\bfG\to\bfGL(V_\Q)$ we obtain a holomorphic embedding of $\cD^+$ in a graded-polarizable mixed period domain $\cM$ as a $\rho(\bfG(\bR)^+\bfU(\C))$-orbit after choosing an integral structure for $V_\Q$ and graded polarization forms.  For a generic $V$ in this orbit we have:
\begin{enumerate}
\item $\MT(V)=\rho(\bfG)$;
\item $\rho(\bfG(\C))\cdot V $ is a closed algebraic subvariety of $\check\cM$;
\item $\rho(\bfG(\bR)^+\bfU(\bC))\cdot V$ is a semialgebraic open subset of $\rho(\bfG(\C))\cdot V $, equal to the component of $(\rho(\bfG(\C))\cdot V) \cap \cM$ containing $V$.
\end{enumerate}

\begin{thm}\label{thm def Hodge var}  Any connected graded-polarizable mixed Hodge variety has the structure of an $\bR_{\alg}$-definable analytic space which is functorial with respect to morphisms of connected mixed Hodge varieties and which agrees with that of $\Gamma\backslash \cM$ from Definition \ref{defn def structure}.
\end{thm}

Before the proof we make some preliminary observations.  For any connected mixed Hodge datum $(\bfG,X_\bfG,\cD^+)$ we define the real-split locus $\cD^+_\bR\subset \cD^+$ as the locus of $h\in \cD^+$ whose Deligne torus is defined over $\bR$, and likewise define the real split locus of any connected mixed Hodge variety as $(\Gamma\backslash\cD^+)_\bR:=\Gamma\backslash\cD^+_\bR$.  Evidently both are $\bR_{\alg}$-definable subspaces and morphisms preserve the real split loci and their definable structures.

For any graded-polarized connected mixed Hodge datum $(\bfG,X_\bfG,\cD^+)$, we have a natural mixed Hodge datum $(\bfH,X_\bfH,\cD^+_{\Gr})$ of the associated graded.  As in section \ref{sect period domains}, we have a natural semi-algebraic identification 
\begin{equation}\cD^+_\bR\cong \cD_{\Gr}^+\times \cS(W\frak{g})(\bR)\label{eq product}\end{equation} where $\cS(W\frak{g})$ is the variety of splittings of the weight filtration of the Lie algebra $\frak{g}$ of $\bfG$, since $\bfU(\bR)=\exp(W_{-1}\frak{g}_\bR)$ acts simply transitively on $\cS(W\frak{g})(\bR)$ by \cite[Prop. 2.2]{cks}.

\begin{prop} The real split locus $(\Gamma\backslash \cD^+)_\bR$ of any connected graded-polarizable mixed Hodge variety admits a structure of a $\bR_{\alg}$-definable topological space characterized by the following property:  for any semi-algebraic Siegel set $\mathfrak{S}\subset(\cD_{\Gr}^+)_\bR$ and bounded semi-algebraic $\Sigma\subset \cS(W\frak{g})(\bR) $, the map $\mathfrak{S}\times\Sigma\to \Gamma\backslash\cM_\bR$ is $\bR_{\alg}$-definable.  Moreover, the definable structure is compatible with morphisms of connected mixed Hodge varieties. 
\end{prop}
\begin{proof}The first part is the same as in the proof of Proposition \ref{prop def struct real split}.  As the identification \eqref{eq product} is clearly functorial in morphisms of connected mixed Hodge data, the second statement follows from Theorem \ref{definable structure on arithmetic varieties} and the fact that a bounded set of Deligne gradings $\bfG_m\to\bfG$ is mapped to a bounded set.
\end{proof}
\begin{proof}[Proof of Theorem \ref{thm def Hodge var}]We start by generalizing the $\sl_2$-splitting:
\begin{lemma}  For any connected graded-polarized mixed Hodge domain $\cD^+$ there is a $\bR_{\alg}$-definable $\bfG(\bR)^+$-equivariant retraction $r:\cD^+\to \cD^+_\bR$ which is compatible with morphisms of connected mixed Hodge domains.
\end{lemma}
\begin{proof}A faithful $\Q$-representation $\rho:\bfG\to \bfGL(V_\Q)$ yields an embedding $\iota :\cD^+\to\cM$ into a graded-polarizable mixed period domain and we may pull back the $\sl_2$-retraction $r:\cM\to\cM_\bR$ to $\cD^+$.

It remains to show that the $\sl_2$-retraction commutes with a morphism $\cD^+\to \cD^{\prime +}$ induced by a morphism of mixed Hodge data $\rho:(\bfG,X_\bfG)\to(\bfG',X'_{\bfG'})$.  For any $h\in X_\bfG$, $\rho$ induces a morphism $d\rho:\frak{g}\to\frak{g}'$ of mixed Hodge structures induced by $h$ and $\rho\circ h$.  The Deligne $\delta$-splitting of $\frak{g}$ is $\Ad(e^{-i\delta})\cdot h$ where $\delta\in (L^{-1,-1}_\frak{g})_\bR$ is the unique element for which $T=\Ad(e^{-2 i\delta})\bar T$, where $T$ is the Deligne grading \cite[Prop. 2.20]{cks}.  From the proof of \cite[Prop. 2.2]{cks}, $\delta$ is contained in $\ad\frak{g}_\bR$, in fact in the Lie algebra generated by the weight torus and its conjugate.  Obviously $d\rho(T)$ is the Deligne grading of $\rho\circ h$, and so $d\rho(\delta)$ is the $\delta$ operator for $\frak{g}'$.  As the $\sl_2$-splitting is defined by universal Lie polynomials in $\delta$, the result follows.
\end{proof}

As in Definition \ref{defn def structure}, we endow $\Gamma\backslash\cM$ with a definable structure coming from the definable set $r^{-1}(\Xi)$ for a definable fundamental set $\Xi$ for $(\Gamma\backslash\cM)_\bR$.  By the lemma this definable structure is compatible with morphisms.
\end{proof}

\subsection{(Weakly) special subvarieties}
Briefly, as in \cite{kHodge} we define the collection of weakly special subvarieties of connected mixed Hodge varieties to be the minimal collection which is closed under finite unions, taking connected components, and taking images and preimages under morphisms of mixed Hodge varieties and which contains points.  For an algebraic variety $S$ with an admissible variation of integral graded-polarized mixed Hodge structures $(\cL, \cW, F)$ with monodromy contained in $\Gamma$, we define the weakly special subspaces of $S$ to be the pull-backs of weakly special subvarieties of $\Gamma\backslash\cM$ along the associated period map $\phi:S\to\Gamma\backslash\cM$ with their natural structure as locally closed $\bR_{\an,\exp}$-definable analytic subspaces, by Theorem \ref{definability of period maps}.  From definable GAGA \cite[Theorem 3.1]{bbt1} we conclude:
\begin{cor}Weakly special subspaces of $S$ are algebraic.
\end{cor}

As a concrete example of the corollary, we specifically treat the case of Noether--Lefschetz loci in more detail, and leave the general setup to the reader.  For any $V\in \cM$, define the Noether--Lefschetz locus \[\cNL(V)=\{V'\in\cM\mid \MT(V')\subset\MT(V)\}\subset\cM.\]
and let $\NL(V)\subset \Gamma\backslash\cM$ be the image.  The following is the mixed analog of \cite[Theorem II.C.1]{ggk}; the same proof works with essentially no modification.
\begin{prop}  For $V\in \cM$, let $\bfG:=\MT(V)$ with unipotent
  radical $\bfU$ and let $X_\bfG$ be the $\bfG(\bR)\bfU(\C)$-conjugacy
  class of the Deligne torus of $V$.  Then the component of $\cNL(V)$ passing through $V$ is the connected mixed Hodge domain for $(\bfG,X_\bfG)$ containing $V$.
\end{prop}
\begin{cor}  $\NL(V)\subset \Gamma\backslash\cM$ is a definable analytic subspace.
\end{cor}
\begin{proof}From the proposition and Theorem \ref{thm def Hodge var}, each connected component of $\NL(V)$ is a definable analytic subspace, and it remains to check there are finitely many components.  For $V'\in\cM$, to have $\MT(V')\subset\MT(V)$ we must check if finitely many vectors in finitely many $T^{m,n}(V')$ are Hodge, that is, contained in $F^0T^{m,n}(V')\cap W_0T^{m,n}(V')$.  Thus, $\cNL(V)=\widecheck\cNL(V)\cap \cM$ for a natural algebraic subvariety $\widecheck\cNL(V)\subset\check\cM$.  As $\widecheck\cNL(V)$ intersects a definable fundamental set for $\Gamma\backslash\cM$ in finitely many components, the result follows.
\end{proof}
For any algebraic variety $S$ with an admissible variation of integral graded-polarized mixed Hodge structures $(\cL, \cW, F)$ with monodromy contained in $\Gamma$ and any $s\in S$ we define $\NL_s\subset S$ to be the pull back of $\NL(\cL_s, \cW_s, F_s)\subset\Gamma\backslash \cM$ with its natural structure as a definable analytic subspace.

\begin{cor}$\NL_s\subset S$ is algebraic.
\end{cor}
Recalling the definition of $\Hdg_0^d(S)\subset S$ from the introduction, we deduce in the same fashion:
\begin{cor}$\Hdg_0^d(S)\subset S$ is algebraic.
\end{cor}

\bibliography{biblio.mixed.definable}
\bibliographystyle{plain}

\end{document}